\documentclass{amsart}
\usepackage{forloop,graphicx,chngcntr,amssymb,amsmath,amsthm,amsrefs,bm,mathrsfs,csquotes,enumerate,wasysym,textcomp,stmaryrd}
\def\nc{\newcommand}
\nc\m\mathbb
\nc\inv{\operatorname{inv}}
\nc\Mal{\operatorname{Mal}}
\nc\MalPref{\operatorname{MalPref}}
\nc\sgn{\operatorname{sgn}}
\nc\matched{\operatorname{Matched}}
\nc\match{\operatorname{Match}}
\nc\PrefStruc{\operatorname{PrefStruc}}

\nc\dto[1]{\xrightarrow[#1]{d}}
\nc\eqd{\overset{d}{=}}

\nc\Var{\operatorname{Var}}
\nc\Cov{\operatorname{Cov}}
\nc\lat{\operatorname{Lat}}
\nc\elat{\operatorname{ELat}}
\nc\Res{\operatorname{Res}}
\nc\Fix{\operatorname{Fix}}
\nc\stab{\operatorname{StabMatch}}
\nc\Nbhd{\operatorname{Nbhd}}
\nc\Bnd{\operatorname{Bnd}}
\nc\gr{\operatorname{gr}}
\nc\id{\operatorname{id}}

\nc\LocPref{\operatorname{LocPref}}
\nc\LeftLat{\operatorname{LeftLat}}
\nc\RightLat{\operatorname{RightLat}}
\nc\offset{\operatorname{offset}}
\nc\MIN{\operatorname{MIN}}
\nc\MAX{\operatorname{MAX}}
\nc\CUT{\operatorname{CUT}}
\nc\wfav{\operatorname{WFAV}}
\nc\mfav{\operatorname{MFAV}}
\nc\empt{\operatorname{EMPTY}}
\nc\witness{\operatorname{WITNESS}}

\newcommand{\C}{{C_0}}

\nc\bbb[1]{\textnormal{\textlbrackdbl}{#1}\textnormal{\textrbrackdbl}}
\nc\bbr[1]{\textnormal{\textlbrackdbl}{#1}\rrparenthesis}
\nc\bbl[1]{\llparenthesis{#1}\textnormal{\textrbrackdbl}}
\nc\bbo[1]{\llparenthesis{#1}\rrparenthesis}

\nc\malefemale{\mathrel{\ooalign{$\male$\cr\hidewidth\raise-.25ex\hbox{$\female\mkern4.5mu$}\cr}}}

  \theoremstyle{plain}
    \newtheorem{thm}{Theorem}
    \newtheorem*{thm*}{Theorem}
    \newtheorem{prop}[thm]{Proposition}
    \newtheorem{lemma}[thm]{Lemma}
    \newtheorem{cor}[thm]{Corollary}

  \theoremstyle{definition}
    
  \theoremstyle{remark}
    \newtheorem*{rmk}{Remark}

\begin{document}
\title{Stable Matchings with Correlated Preferences}
\author[Christopher~Hoffman]{ \ Christopher~Hoffman}
	\address{Department of Mathematics,
	University of Washington, Seattle, WA 98195} 
	\email{hoffman@math.washington.edu}
	
	\author[Avi~Levy]
	{\ Avi~Levy}
	\address{Microsoft Corporation, Redmond, WA 98052} \email{avi.levy@microsoft.com}
	
	\author[Elchanan~Mossel]
	{Elchanan~Mossel}
	\address{Department of Mathematics,  
	Massachusetts Institute of Technology, 
	Cambridge, MA 02139} \email{elmos@mit.edu.}

\thanks{This work was supported in part by NSF DMS-1712701, DMS-1954059
ARO MURI W911NF1910217, Bush Faculty Fellowship ONR-N00014-20-1-2826 and 
Simons Investigator award (622132)}

%\date{\today}

\maketitle

 \begin{abstract}
 
 The stable matching problem has been the subject of intense theoretical and empirical study since the seminal 1962 paper by Gale and Shapley \cite{gale1962college}.  The number of stable matchings for different systems of preferences has been studied in many contexts, going back to Donald Knuth in the 1970s.
 In this paper, we consider a family of distributions defined by the Mallows permutations and show that with high probability the number of stable matchings for these preferences is exponential in the number of people.
  \end{abstract}

  \section{Introduction}

  Suppose there are $n$ men and $n$ women, each of whom ranks every member of the opposite gender in such a way that there are no ties.
  A matching (in which every woman is partnered with a distinct man) is stable if there does not exist a woman and a man, each of whom prefers the other to their partner in the matching.

  The Gale-Shapley algorithm \cite{gale1962college} constructs a stable matching for any set of rankings, proving the non-trivial result that there always exists such a matching. This is a fundamental result in the economic theory of stable allocations, a theory which appeared in the citation of the 2012 Nobel Prize in Economics awarded to Lloyd Shapley and Alvin E. Roth.

 \subsection{Number of Matchings}
 
  Two natural questions about the number of stable matchings are
 \begin{itemize} 
\item What is the largest possible number of matchings that a preference system with $n$ men and $n$ women can have?
\item   How many stable matchings do typical preference systems have?
\end{itemize}

There is a system with 2 men and 2 women that has 2 stable matchings.  It is easy to extend this to get a system with $2n$ men and $2n$ women with $2^{n}$ stable matchings. In their book on stable marriage \cite{gusfield1989stable}, Gusfield and Irving show that there are preferences (of $n$ men and $n$ women) with $2^n$ stable matchings. 
 This has been improved to show that the lower bound is at least approximately $2.28^n$ \cite{leather1986complexity}.  
  One expects this type of exponential growth for rankings that are globally, but not locally, correlated. The basic idea is that the stable matchings decompose into approximately independent blocks of constant size, and each block carries a constant number of stable matchings. 

The trivial upper bound on the number of stable matchings is $n!$. A recent result of
Karlin, Oveis Gharan and Weber has improved this to prove that there exists a constant $\C<\infty$ such that for any $n$ and any set of preferences the number of stable matchings is upper bounded by $\C^n$ \cite{karlin2018simply}.

The question about typical systems has typically been studied when the rankings are independent and uniformly distributed and as $n\to\infty$. In this case Pittel proved that the expected number of stable matchings is asymptotic to $e^{-1}n\log n$ \cite{pittel1989average}. He also proved that there are at least $n^{1/2-o(1)}$ stable matchings with high probability \cite{pittel1992likely}. Lennon and Pittel proved that the probability that there are at least $cn\log n$ stable matchings is bounded below by
  $$
    \frac{2e}{2e+1}\approx .84
  $$
  in the limit as $c\to 0$ \cite{lennon2009likely}. 
  It remains unknown whether there exists a constant $C$ such that the number of stable matchings is asymptotic to $Cn\log n$ with high probability.

Ashlagi, Kanoria and  Leshno considered unbalanced systems where, for example, there are $n$ men and $n+1$ women and preferences are independent and uniform. They proved in this situation
most men and women have a unique stable match and that the men rate their unique partner significantly higher than the women rate their partner.
\cite{Ashlagi2017unbalanced}. Pittel proved that the expected number of stable matches in such a system is $n/(e \log n)$ \cite{pittel2019likely}.

While these results about independent and uniform preferences are mathematically interesting, the choice of independent uniform preferences is far from what is actually observed in the real world. (For example, if a college applicant is ranked highly by Harvard then she is probably also ranked highly by Yale as well.)

In this paper we consider preference structures which are not uniform. 
A classical model for correlated permutations is given by the 
\textit{Mallows}-distribution~\cite{mallows1957non}. 
This model has been extensively used and analyzed as a model of data, see e.g.~\cite{braverman2009sorting,lu2014effective,liu2018efficiently} and follow up work. 
 We show that in this setting we get very different behavior than what Pittel observed. the expected number of stable matchings is exponential in the length of the permutation.

\subsection{Notation}
We use the following notation to describe the model formally. 
  An \textbf{interval of integers} is a set of the form $I\cap \m Z$ where $I\subseteq \m R$ is an interval. 
  We write $\bbr{a,b}:=[a,b)\cap \m Z$ and similarly for other intervals.

  For a (possibly infinite) interval $I\subseteq \m Z$, let $\Omega_I:=I\times \{\male,\female\}$ and let $S_I$ be the set of permutations of $I$ (i.e., bijections from $I$ to itself). 
  A \textbf{preference structure} $\mathcal P$ on $\Omega_I$ is a map from $\Omega_I$ to $S_I$. 
  Let $\PrefStruc_I:=(S_I)^{\Omega_I}$ denote the set of preference structures on $\Omega_I$. 
  %Let $\Omega_n:=\{1,\ldots,n\}\times \{\male,\female\}$ and let $S_n$ be the set of permutations of $\{1,\ldots,n\}$. 
  %A \textbf{preference structure} $\mathcal P$ on $\Omega_n$ is a map from $\Omega_n$ to $S_n$. 
  For $i\in I$, we regard $\mathcal P\bigl((i,\female)\bigr)$ as a ranking of $I\times \{\male\}$, and we say that $(i,\female)$ \textbf{prefers} $(j,\male)$ to $(k,\male)$ if
  $$
    \mathcal P\bigl((i,\female)\bigr)(j)>\mathcal P\bigl((i,\female)\bigr)(k).
  $$
  Similarly, we say that $(i,\male)$ prefers $(j,\female)$ to $(k,\female)$ if
  $$
    \mathcal P\bigl((i,\male)\bigr)(j)>\mathcal P\bigl((i,\male)\bigr)(k).
  $$ 

  A \textbf{match} is an unordered pair $\{(i,\male),(j,\female)\}$ with $i,j\in I$, and the \textbf{partner} of an element of a match is the other element.
  A \textbf{matching} of $\Omega_I$ is a set of matches for which every element of $\Omega_I$ appears in exactly one match.
  (That is, we consider only perfect matchings, and not partial matchings.)
  Finally, a matching is \textbf{stable} with respect to $\mathcal P$ if there do not exist integers $i$ and $j$ such that $(i,\male)$ and $(j,\female)$ prefer one another to their respective partners in the matching. 
  When there is no risk of confusion, we refer to a matching of $\Omega_I$ that is stable with respect to $\mathcal P$ as a \textbf{stable matching of $\mathcal P$}. We write $\stab\mathcal P$ for the set of stable matchings of $\mathcal P$.
  In the special case when $I=\bbb{1,n}$, we write $S_n$, $\Omega_n$, and $\PrefStruc_n$ in place of $S_I$, $\Omega_I$, and $\PrefStruc_I$, respectively.

  Each ranking is identified with a permutation $\pi$ of $\{1,\ldots,n\}$ by declaring $\pi(i)<\pi(j)$ if $j$ is more desirable than $i$. 
  The Mallows measure $\Mal=\Mal_{q,n}$ with parameter $q\in (0,1)$ assigns to each $\pi$ a probability proportional to $q^{\inv(\pi)}$, where $\inv(\pi)$ is the \textbf{inversion number}
  \begin{equation} \label{escalation}
    \inv(\pi)=\#\bigl\{1\leq i<j\leq n\colon \pi(i)>\pi(j)\bigr\}.
  \end{equation}
  This probability is maximized for the identity permutation, and decays exponentially in the inversion number; thus a Mallows random permutation may be regarded as a perturbation of the identity.
  Hence we regard the rankings in our model as being `globally correlated' with the identity permutation, even though the rankings are independent of one another.

\subsection{Statement of Main Results}
  For a parameter $q\in (0,1)$ and an integer $n\geq 1$, let $\MalPref_{q,n}$ be a random element of $\PrefStruc_n$ assigning independent $\Mal_q$-distribution permutations to all individuals in $\Omega_n$. 
  In other words, the law of $\MalPref_{q,n}$ is the product measure $(\Mal_q)^{\Omega_n}$. We refer to $\MalPref_{q,n}$ as being a \textbf{Mallows preference structure}. 
  %we define a \textbf{Mallows preference structure} $\MalPref=\MalPref_{q,n}$ to be a random preference structure assigning to each individual in $\Omega_n$ an independent $\Mal_q$-distributed permutation. 
  We prove that the number of stable matchings of $\MalPref_{q,n}$ grows exponentially in $n$.
  
  \begin{thm}\label{thm:exp}
    For all $q\in (0,1)$, there exists a constant $0<\gr_q<\infty$ such that
    $$
      \frac{\log\#\stab\MalPref_{q,n}}{n}\dto{n\to\infty}\gr_q.
    $$
    Furthermore, $\gr$ is a continuous function such that $\gr_q\to 0$ as $q\to 0$.
  \end{thm}

  The difficulty in analyzing stable matchings is that they are globally
  defined. The goal of our proof is to show that 
 given a preference structure on $\Omega_n$, we define an equivalence relation decomposing $\bbb{1,n}$ into intervals of integers in such a way that the stable matchings on the larger interval are determined by those on each equivalence class. To do so we define an infinite version of the matching problem which allows us to apply Ergodic Theory and then show how results for the infinite problem imply the finite $n$ result. 

The rest of the paper is structured as follows. In Section \ref{sec:ergo} we define the Mallows measure on permutations of ${\m Z}$ and we obtain results about the stable matchings by coupling preference structures on $\Omega_n$ with preference structures on $\Omega_{\m Z}$.    The latter is shift invariant which allows us to use the ergodic theorem. We also define the notion of an essential lattice cutpoint.  This notion will be critical to our arguments.
In Section \ref{sec:cutpoint} we give a criteria which allows us to break up a preference structure on $\Omega_n$ to preference structures on subintervals. In Section \ref{sec:positive} we show that this criteria works well with the Mallows measure and breaks the interval $\bbb{1,n}$ into linearly many subintervals. In Section \ref{sec:expo} we combine these results to prove our main theorem.

  %We will deduce this theorem from results in the next three sections. 
  %In Section~\ref{sec:ergo} we couple $(\MalPref_{q,n})_{n\geq 1}$ with an ergodic system, we define a set of cutpoints at which the stable matchings decompose, and we apply the ergodic theorem to a local approximation of the growth rate. In Section~\ref{sec:cutpoint} we present a criterion for locating these cutpoints. In Section~\ref{sec:positive}, we show that two key probabilities are positive for all $q\in (0,1)$, with quantitative lower bounds. Finally in Section~\ref{sec:expo} we prove the main theorem by showing that the local approximation to the growth rate is integrable, with mean decaying to $0$ as $q\to 0$.

  %Since the Mallows model with $q=1$ is uniformly random, the previous results by Pittel et. al. imply that $\gr_1=0$. 
  %We do not know if the growth rate $\gr_q$ is continuous at $q=1$ (or anywhere else, besides at $q=0$). 
  %However, it is conceivable that there exists a $q^*\in (0,1)$ at which $\gr_q$ attains its maximum. 
  %This would follow from the stronger claim that $\gr_q$ is a continuous function of $q$ on $[0,1]$. 

  %%%%%%%%%%%%%%%
  %%%%%%%%%%%%%%%
  %% Coupling  %%
  %%%%%%%%%%%%%%%
  %%%%%%%%%%%%%%%
  
  \begin{rmk}
  The result of the current paper are part of Avi Levy's Ph.D. thesis~\cite{levy2017novel}. 
  \end{rmk}

  \section{Coupling with an ergodic system}\label{sec:ergo}

  In this section we couple the sequence of finite random preference structures $(\MalPref_{q,n})_{n\geq 1}$ with a random infinite preference structure.
  The latter is stationary and ergodic, and applying the ergodic theorem to it will imply the existence of the constant $\gr_q$ in the main theorem.

  For a (finite or infinite) set of integers $U\subseteq \m Z$, a permutation of $U$ is defined to be a bijection from $U$ to itself.
  %We write $S_{U}$ for the set of permutations of $U$.
  If $T$ is a finite subset of $U$ and $\pi$ is a permutation of $U$, the permutation $\pi_T$ \textbf{induced} by $T$ is the unique permutation of $T$ satisfying
  $$
    \pi_T(s)<\pi_T(t)\iff \pi(s)<\pi(t),\qquad s,t\in T.
  $$
  In other words, the elements of $T$ are given the same relative rankings under $\pi$ and $\pi_T$.

  The Mallows measure $\Mal_{q,I}$ on permutations of a finite interval $I\subset\m Z$ assigns a probability proportional to $q^{\inv(\pi)}$ to each permutation $\pi$ of $I$, where  $\inv(\pi)$ was defined in \eqref{escalation}.
%  $$
%    \inv(\pi):=\#\{i,j\in I\colon i<j,\ \pi(i)>\pi(j)\}.
%  $$
  %Let $\Omega_I:=I\times \{\male,\female\}$. 
  %A preference structure on $\Omega_I$ is a map from $\Omega_I$ to the set of permutations of $I$. 
  We write $\MalPref_{q,I}$ for a random preference structure assigning independent $\Mal_{q,I}$-distributed permutations to the elements of $\Omega_I$.

  There is an infinite extension of the Mallows measure having the following properties.

  \begin{thm*}[\cite{gnedin2012two}]
    For all $q\in [0,1)$ there exists a measure $\Mal_{q,\m Z}$ on permutations of $\m Z$ such that if $\pi$ has law $\Mal_{q,\m Z}$, then for all finite intervals $I\subset \m Z$ the induced permutation $\pi_I$ has law $\Mal_{q,I}$. 
    Moreover, the following stationarity property holds:
    $$
      \bigl(\pi(i)\bigr)_{i\in\m Z}\eqd \bigl(\pi(i+1)-1\bigr)_{i\in\m Z}.
    $$
  \end{thm*}

  %Let $\Omega_{\mathbb Z}:=\mathbb Z\times \{\male,\female\}$. A preference structure $\mathcal P$ on $\Omega_{\mathbb Z}$ is a map from $\Omega_{\mathbb Z}$ to permutations of $\mathbb Z$. 
  We define $\MalPref_{q,\mathbb Z}$ to be a random preference structure assigning to each individual in $\Omega_{\m Z}$ an independent $\Mal_{q,\m Z}$-distributed permutation. 
  The law of $\MalPref_{q,\m Z}$ is $(\Mal_{q,\m Z})^{\Omega_{\m Z}}$.

  Let $I\subset \m Z$ be a finite interval of integers. 
  For a preference structure $\mathcal P$ on $\Omega_{\m Z}$, we define the preference structure $\mathcal P_{I}$ \textbf{induced} by $I$ to be the preference structure of $\Omega_{I}$ given by
  $$
    \mathcal P_{I}\bigl((k,\malefemale)\bigr)=\mathcal P\bigl((k,\malefemale)\bigr)_{I},\qquad k\in I,\ \malefemale\in\{\male,\female\}.
  $$
  
  Observe that for all finite intervals $I\subset \m Z$,
  $$
    \bigl(\MalPref_{q,\m Z}\bigr)_{I}\eqd\MalPref_{q,I}.
  $$
  In particular, taking $I=\bbb{1,n}$ yields a coupling of $(\MalPref_{q,n})_{n\geq 1}$ with $\MalPref_{q,\m Z}$.

  Let $\Theta$ map $\PrefStruc_{\m Z}$ to itself by relabeling person $i$ to person $i+1$. 
  More precisely,
  $$
    \Theta\mathcal P\bigl((i,\malefemale)\bigr)(j)=\mathcal P\bigl((i+1,\malefemale)\bigr)(j+1)-1,\qquad i,j\in\m Z,\ \malefemale\in\{\male,\female\}.
  $$

  %That is, man $i$ and woman $i$ in $\Theta\mathcal P$ are assigned the shifted preference lists of man $i+1$ and woman $i+1$ in $\mathcal P$, respectively.
  \begin{lemma}\label{lem:erg}
    The triple $\bigl(\PrefStruc_{\m Z},\Theta,(\Mal_q)^{\Omega_{\m Z}}\bigr)$ is an ergodic measure-preserving system.
  \end{lemma}
  That the system is measure-preserving follows from stationarity of $\Mal_{q,\m Z}$. The proof of ergodicity is a variant of the well-known proof of Kolmogorov's $0$-$1$ law \cite{kallenberg2006foundations}.  
  \begin{proof}
    Let $\mathcal P$ be a random preference structure with law $(\Mal_{q,\m Z})^{\Omega_{\m Z}}$. 
    Let $\mathcal I$ be the invariant $\sigma$-algebra, consisting of all events $A\subseteq \PrefStruc_{\m Z}$ satisfying $A=\Theta^{-1}A$. 
    We wish to show that $\mathcal I$ is trivial. For a finite interval of integers $I\subset \m Z$, consider the $\sigma$-algebra
    $$
      \mathcal F_I=\sigma\Bigl(\mathcal P\bigl((i,\malefemale)\bigr)\colon i\in I,\ \malefemale\in\bigl\{\male,\female\bigr\}\Bigr).
    $$
    Fix $n\geq 1$. 
    If $A\in \mathcal F_{\bbb{-n,n}}$, then for all $i\in\m Z$ we have that $\Theta^iA\in \mathcal F_{\bbb{-n+i,n+i}}$. 
    In particular, if $i>2n$ then $A$ is independent of $\Theta^iA$. 
    Thus the $\sigma$-algebra $\mathcal I\cap \mathcal F_{\bbb{-n,n}}$ is trivial. 
    The result now follows by taking $n\to\infty$ and using a standard approximation argument. 
  \end{proof}

    % The random variable $\MalPref_{q,\m Z}$ with law $(\Mal_q)^{\Omega_{\m Z}}$ is stationary and ergodic with respect to $\Theta$, and for any factor map $F$, its image $F(\MalPref_{q,\m Z})$ in $\mathbb R^{\m Z}$ is also stationary and ergodic.
  
  A map $F\colon \PrefStruc_{\m Z}\to \m R^{\m Z}$ satisfying
  $$
    F(\Theta\mathcal P)_i=F(\mathcal P)_{i+1},\qquad i\in \m Z,\ \mathcal P\in\PrefStruc_{\m Z},
  $$
  is called a \textbf{factor map}. 
  It intertwines $\Theta$ with the shift map on $\mathbb R^{\m Z}$.

  \begin{cor}\label{cor:prefStatErg}
    If $F\colon \PrefStruc_{\m Z}\to \m R^{\m Z}$ is a factor map and $\MalPref_{q,\m Z}$ is a random variable with law $(\Mal_q)^{\Omega_{\m Z}}$, then the random sequence $F(\MalPref_{q,\m Z})$ is stationary and ergodic.% with respect to the shift.
  \end{cor}
  \begin{proof}
    This is an immediate consequence of the previous lemma.
  \end{proof}

  %The map $\Theta$ preserves the measure $(\Mal_q)^{\Omega_{\m Z}}$, and is ergodic with respect to this measure as well. 
  %Furthermore, $\MalPref_{q,\m Z}$ is ergodic with respect to $\Theta$.
  %Indeed, the latter claim follows from Kolmogorov's $0$-$1$ law \cite{kallenberg2006foundations}. % since $\MalPref_{q,\m Z}$ is an iid sequence of $\Mal_{q,\m Z}$-distributed permutations.

  %(i.e., intertwining $\Theta$ and the shift), the sequence $F(\MalPref_{q,\m Z})$ is stationary and ergodic. 
  
  Combining the pointwise ergodic theorem (see \cite{kallenberg2006foundations}) with Lemma~\ref{lem:erg} implies the following. For any (product Borel) measurable function $f\colon \PrefStruc_{\m Z}\to \m R$ for which $f(\MalPref_{q,\m Z})$ is integrable, we have that
  \begin{equation}\label{eq:flim}
    \lim_{n\to\infty}\frac{1}{n}\sum_{i=1}^nf\bigl(\Theta^i\MalPref_{q,\m Z}\bigr)=\mathbb Ef\bigl(\MalPref_{q,\m Z}\bigr)\qquad a.s.
  \end{equation}
  In particular, taking $f=\bm{1}[A]$ for $A$ a measurable subset of $\PrefStruc_{\m Z}$ yields that
  \begin{equation}\label{eq:lim}
    \lim_{n\to\infty}\frac{\#\bigl\{1\leq i\leq n\colon \Theta^i\MalPref_{q,\m Z}\in A\bigr\}}{n}
    =\mathbb P\bigl(\MalPref_{q,\m Z}\in A\bigr)\qquad a.s.
  \end{equation}

\subsection{Lattice Cutpoints}  Preference structures whose stable matchings decompose at a given location play a crucial role in our argument. 
  Fix integers $i$, $j$, and $k$ with $i\leq j\leq k$ and a preference structure $\mathcal P$ on $\Omega_{\bbb{i,k}}$. 
  We say that $j$ is a \textbf{lattice cutpoint} of $\mathcal P$ if the set of matchings of $\Omega_{\bbb{i,k}}$ that are stable with respect to $\mathcal P$ coincides with the set of matchings arising as the union of a stable matching of $\mathcal P_{\bbb{i,j}}$ with a stable matching of $\mathcal P_{\bbl{j,k}}$. 
  
  Lattice cutpoints are our tool for counting the number of stable matchings, since if $i\leq j\leq k$ are integers and $j$ is a lattice cutpoint of $\mathcal P_{\bbb{i,k}}$ then
  \begin{equation}\label{eq:factorStab}
    \#\stab\mathcal P_{\bbb{i,k}}=\#\stab\mathcal P_{\bbb{i,j}}\cdot\#\stab\mathcal P_{\bbl{j,k}}.
  \end{equation}
  (In fact, there is a natural lattice structure on the set of stable matchings \cite{blair1988lattice}, and \eqref{eq:factorStab} generalizes to a lattice isomorphism between $\stab\mathcal P_{\bbb{i,k}}$ and the product of the lattices $\stab\mathcal P_{\bbb{i,j}}$ and $\stab\mathcal P_{\bbl{j,k}}$; hence the name.) 

  For a (possibly infinite) interval $I\subseteq \m Z$ and for a preference structure $\mathcal P$ on $\Omega_I$, we say that $j\in I$ is an \textbf{essential lattice cutpoint} of $\mathcal P$ if it is a lattice cutpoint of $\mathcal P_{\bbb{i,k}}$ for all $i,k\in I$ with $i\leq j\leq k$. 
  Note that only finite preference structures appear in the definition.
  Denote by $\elat(\mathcal P)$ the set of essential lattice cutpoints of $\mathcal P$. 

  The random set $\elat(\MalPref_{q,\m Z})$ plays a fundamental role in our argument. 
  We define 
  $$
    \rho_q:=\mathbb P\bigl(0\in \elat( \MalPref_{q,\m Z})\bigr).
  $$
  Observe that $\elat(\MalPref_{q,\m Z})$ is stationary and ergodic. 
  \begin{lemma}\label{lem:latDense}
    For all $q\in (0,1)$,
    $$
      \lim_{n\to\infty}\frac{\#\bigl(\elat(\MalPref_{q,\m Z})\cap \bbb{1,n}\bigr)}{n}=\rho_q\qquad a.s.
    $$
  \end{lemma}

  \begin{proof}
    Apply the ergodic theorem in the form \eqref{eq:lim}, with $$A=\bigl\{0\in \elat(\mathcal \MalPref_{q,\m Z})\bigr\}.$$
  \end{proof}

  Essential lattice cutpoints allow us to define the following local quantities, for a preference structure $\mathcal P$ on $\Omega_{\m Z}$ and an integer $i$. 
  Let $\Nbhd(\mathcal P,i):=\bbl{\ell,r}$ where $\ell$ and $r$ are largest and smallest elements of the sets
  $$
    \bbr{-\infty,i}\cap \bigl(\elat(\mathcal P)\cup \{-\infty\}\bigr)\qquad\text{and}\qquad \bbb{i,\infty}\cap \bigl(\elat(\mathcal P)\cup \{\infty\}\bigr),
  $$
  respectively. 
  That is, $\Nbhd(\mathcal P,i)$ is the interval whose endpoints are the essential lattice cutpoints surrounding $i$, inclusive on the right and exclusive on the left. We write $\LocPref(\mathcal P,i)$ for the induced preference structure $\mathcal P_{\Nbhd(\mathcal P,i)}$, and we set
  \begin{equation}\label{eq:fP}
    \gr(\mathcal P,i):=\frac{\log\#\!\stab\LocPref(\mathcal P,i)}{\# \Nbhd(\mathcal P,i)},
  \end{equation}
  provided $\Nbhd(\mathcal P,i)$ is finite, and otherwise $\gr(\mathcal P,i):=0$. 
  This quantity is a local approximation at $i$ to the exponential growth rate of the number of stable matchings. 
  Observe that $\gr(\mathcal P,i)=\gr(\Theta^{-i}\mathcal P,0)$, from which it follows by earlier general remarks that the sequence
  $
    \bigl(\gr(\MalPref_{q,\m Z},i)\bigr)_{i\in\m Z}
  $
  is stationary and ergodic. 
  We also point out that by \cite{karlin2018simply} there exists $\C<\infty$ such that  
  \begin{equation}\label{eq:trivBndGr}
    \gr(\mathcal P,i)\leq \C.
  \end{equation}

  The next result is crucial for the proof of the main theorem.
  \begin{prop}\label{prop:main}
    For all $q\in (0,1)$ 
    such that %$\gr\bigl(\MalPref_{q,\m Z},0\bigr)$ is integrable and 
    $\rho_q>0$, 
    $$
      \frac{\log\#\stab\MalPref_{q,n}}{n}\dto{n\to\infty}\mathbb E \gr(\MalPref_{q,\m Z},0).
    $$
  \end{prop}
  Thus, the constant $\gr_q$ in the main theorem is given by $\mathbb E \gr(\MalPref_{q,\m Z},0)$, whenever the above hypotheses hold. 
  Once we show that $\rho_q>0$ and $0<\gr_q<\infty$ for all $q\in (0,1)$ and that these functions are continuous, the main theorem will follow. 
  The proofs of these assertions are non-trivial and they will occupy the remaining sections.
  \begin{proof}[Proof of Proposition~\ref{prop:main}]
    Fix $q\in (0,1)$ satisfying the hypotheses. 
    Let $\mathcal P=\MalPref_{q,\m Z}$, let $X_n=\log\#\!\stab\mathcal P_{\bbb{1,n}}$, and let $Y_n=\sum_{i=1}^{n}\gr(\mathcal P,i)$. 
    By the ergodic theorem, $\lim_{n\to\infty}\tfrac{1}{n}Y_n$ converges to a constant a.s. We define $\gr_q=\lim_{n\to\infty}\tfrac{1}{n}Y_n$ to be this a.s.\ limit.

    We show that $\tfrac{1}{n}|Y_n-X_n|$ converges to $0$ in probability. 
    Let $$N=\min \bigl(\elat(\mathcal P)\cap \bbr{1,\infty}\bigr).$$ 
    Define the random variables $(M_n)_{n\geq 1}$ via
    $$
      M_n:=\begin{cases}
        n-\max\bigl(\elat(\mathcal P)\cap \bbb{1,n}\bigr),& N\leq n\\
        0,& N>n.
      \end{cases}
    $$
    On the event $N\leq n$, the integers $N$ and $n-M_n$ are the minimal and maximal elements of $\elat(\mathcal P)\cap\bbb{1,n}$, respectively.  
    The quantity $\gr(\mathcal P,i)$ is constant for $i\in\bbb{1,N}$ and for $i\in \bbl{n-M_n,n}$. 
    Set $X_n'=\log\#\!\stab\mathcal P_{\bbl{N,n-M_n}}$. 
    By \eqref{eq:factorStab},
    \begin{align*}
      X_n&=\log\#\!\stab\mathcal P_{\bbb{1,N}}+X_n'+\log\#\!\stab\mathcal P_{\bbl{n-M_n,n}}\\
      Y_n&=N\gr(\mathcal P,1)+X_n'+M_n \gr(\mathcal P,n).
    \end{align*}
%    Let $f(x)=x\log x$ for $x\geq 1$. 
    By the last equations and \eqref{eq:trivBndGr},
    \begin{equation}\label{eq:diffBound}
      |X_n-Y_n|\leq CN+CM_n 
%      %+\max\BiglC\#\Nbhd(\mathcal P,1)\bigr),f\bigl(\#\Nbhd(\mathcal P,n)\bigr)\Bigr)
%    \end{equation*}
%    As $N\leq \# \Nbhd(\mathcal P,1)$ and $M_n\leq \# \Nbhd(\mathcal P,n)$ and $f$ is increasing, it follows that
%    \begin{equation}\label{eq:diffBound2}
%      |X_n-Y_n|
      \leq 2\max\Bigl(\C\bigl(\#\Nbhd(\mathcal P,1)\bigr),\C\bigl(\#\Nbhd(\mathcal P,n)\bigr)\Bigr).
    \end{equation}
    While the bound \eqref{eq:diffBound} was obtained on the event $N\leq n$, it holds trivially on the event $N>n$ as well. 
    Thus for any $\epsilon>0$, we have that
    \begin{align}
      \mathbb P\bigl(|X_n-Y_n|>\epsilon n\bigr)
        &\leq \mathbb P\Bigl(\C\bigl(\#\Nbhd(\mathcal P,1)\bigr)>\tfrac{\epsilon n}{2}\Bigr)
        +\mathbb P\Bigl(\C\bigl(\#\Nbhd(\mathcal P,n)\bigr)>\tfrac{\epsilon n}{2}\Bigr)\nonumber\\
      &=2\mathbb P\Bigl(\C\bigl(\#\Nbhd(\mathcal P,1)\bigr)>\tfrac{\epsilon n}{2}\Bigr).\label{eq:XnYn}
    \end{align}
    Combining the hypothesis $\rho_q>0$ with Lemma~\ref{lem:latDense}, we see that $\#\Nbhd(\mathcal P,1)$ is a.s.\ finite. 
    Thus as $n\to\infty$, the probability in \eqref{eq:XnYn} tends to $0$, implying that $\tfrac{1}{n}|X_n-Y_n|$ converges to $0$ in probability as $n\to\infty$. 
    When combined with the a.s.\ convergence of $\tfrac{1}{n}Y_n$ to $\gr_q$, it follows that $\tfrac{1}{n}X_n$ converges in distribution to $\gr_q$.
  \end{proof}

  %%%%%%%%%%%%%%%
  %%%%%%%%%%%%%%%
  %% Cutpoint  %%
  %%%%%%%%%%%%%%%
  %%%%%%%%%%%%%%%

  \section{Cutpoint criteria}\label{sec:cutpoint}

    Essential lattice cutpoints of preference structures were introduced for the purpose of splitting the task of enumerating stable matchings into smaller independent parts. 
    A preference structure $\mathcal P\in \PrefStruc_{\m Z}$ induces an  equivalence relation on $\m Z$ by taking the equivalence classes to be the intervals between consecutive essential lattice cutpoints; this is the equivalence relation referred to in the introduction. 
    For any finite interval of integers $I$ that is a union of equivalence classes, the number of stable matchings of the induced preference structure $\mathcal P_I$ factors as a product of the number of stable matchings on each of the equivalence classes contained in $I$. 
    Thus if we can locate essential lattice cutpoints, we will be able to relate the exponential growth rate of the number of stable matchings to a local quantity.

    %For this essential lattice cutpoints to be useful, we need a way of locating them that does not require \textit{a priori} knowledge of stable matchings. 
    We show that essential lattice cutpoints can be detected by a certain family of inequalities for the preference structure.  
    In fact, we start by showing that simpler inequalities yield a type of cutpoint weaker than the essential lattice cutpoint, defined as follows. 
    %The weaker notion of cutpoint is defined as follows. 
    Given integers $i\leq j\leq k$ and a matching $M$ of $\Omega_{\bbb{i,k}}$, we say that $j$ is a \textbf{cutpoint} of $M$ if there do not exist integers $a$ and $b$ with $i\leq a\leq j$ and $j+1\leq b\leq k$ such that either $(a,\female)$ and $(b,\male)$ are matched, or $(a,\male)$ and $(b,\female)$ are matched. 
    It follows from the definitions that if $j$ is an essential lattice cutpoint of a preference structure $\mathcal P$ on some interval $I\ni j$, then for all $i,k\in I$ with $i\leq j\leq k$ we have that $j$ is a cutpoint of the induced preference structure $\mathcal P_{\bbb{i,k}}$. %on $\Omega_{\bbb{i,k}}$, then $j$ is a cutpoint of every stable matching of the preference structure. Thus, this is indeed a weaker notion of cutpoint.

  \begin{lemma}\label{lem:cutBound}
    Let $I\subset \m Z$ be a finite interval, let $\mathcal P$ be a preference structure on $\Omega_I$, and let $s\in\m Z+\tfrac12$ be a half-integer. 
    If for all $i,j\in I$ and both \textup{$\malefemale\in \{\male,\female\}$}
    \begin{equation}\label{eq:bndForCut}
      \Bigl|\mathcal P\bigl((i,\textup{$\malefemale$})\bigr)(j)-j\Bigr|<\frac{|i-s|+|j-s|}{20},
    \end{equation}
    then $\lfloor s\rfloor$ is a cutpoint of every stable matching of $\mathcal P$.
  \end{lemma}

  We say that $\mathcal P$ satisfies the \textbf{cutpoint bound} at location $s$ whenever \eqref{eq:bndForCut} holds for all $i,j\in I$ and both $\malefemale\in \{\male,\female\}$. 
  Before presenting the proof, we present a simple result.

  % We say that $\mathcal P$ satisfies the \textbf{cutpoint bound} at location $k+\tfrac12$ whenever \eqref{eq:bndForCut} holds for all $i$, $j$, and $\pi$ as above. Before proving the lemma, we record the following simple result.
  \begin{lemma}\label{lem:crossOver}
    Let $I\subset \m Z$ be a finite interval and let $M$ be a (perfect) matching of $\Omega_I$. 
    Given a match \textup{$\{(a,\male),(b,\female)\}$} in $M$ with $a<b$ and given any real number $t\in (a,b)$, there exist integers $a'$ and $b'$ with $a'<t<b'$ such that the match \textup{$\{(a',\female),(b',\male)\}$} is also in $M$.
  \end{lemma}
  \begin{proof}
    If no such $a'$ and $b'$ existed, then the matching would furnish a bijection between sets whose cardinalities differ by one. This is impossible.
  \end{proof}

  The proof of Lemma~\ref{lem:cutBound} is rather subtle. 
  In words, the cutpoint bound \eqref{eq:bndForCut} says that preferences do not stray too far from being in order, with the amount of variation increasing at most linearly in the distance from the purported cutpoint location. 
  Therefore if a long match appears in some stable matching, it is not caused by exotic preferences but rather by an even longer match going in the opposite direction and involving a person even higher up in the ranking. 
  Since we are working with a finite system, this cannot go on forever. 
  Therefore there are, in fact, no matches across the purported cutpoint.

%  Here `longness' is measured relative to the distance from the cutpoint location. 
%  In fact, finding the correct definition is a delicate balancing act. 
%  Fortunately this task rests with the authors rather than with the reader, and in the end the definition is not too complicated.

  \begin{proof}[Proof of Lemma~\ref{lem:cutBound}]
    Fix $I$, $s$, and $\mathcal P$ as in the lemma, as well as a stable matching $M$ of $\mathcal P$. 
    We claim that the set
    $$
    \mathcal S:=\left\{\begin{aligned}
      (i,j)&\in I\times I\colon \text{$i>s$ and $i-s>3(j-s)$ and either}\\
      &\text{$(i,\female)$ and $(j,\male)$ are matched, or $(i,\male)$ and $(j,\female)$ are matched}
    \end{aligned}\right\}
    $$
    is empty, from which the result will follow. That is, there are no long matches.

    Suppose to the contrary that $\mathcal S$ is non-empty. Since $\mathcal S$ is finite, we may then select a pair $(i,j)\in\mathcal S$ with $i$ maximal. Suppose, for the sake of concreteness, that it is $(i,\female)$ which is matched with $(j,\male)$, rather than the other way around (the argument is unmodified apart from interchanging the roles of $\male$ and $\female$). The equation $i-s>3(j-s)$ is equivalent to $j<\tfrac13(i+2s)$, and therefore when $(i,j)\in\mathcal S$ we have that 
    \begin{equation}\label{eq:jBound}
      j<\tfrac13 i+\tfrac23 s<\tfrac23i+\tfrac13s<i,
    \end{equation}
    the latter two inequalities being equivalent to $i>s$. Applying Lemma~\ref{lem:crossOver} with $t=\tfrac23i+\tfrac13s$, it follows that there exist integers $i'$ and $j'$ such that $(i',\male)$ is matched with $(j',\female)$ and
    \begin{equation}\label{eq:primeChoice}
      j'<\tfrac23i+\tfrac13 s<i'.
    \end{equation}
    In particular, it follows from \eqref{eq:primeChoice} that $i'>s$.

    We claim that $\pi=\mathcal P\bigl((i,\female)\bigr)$ ranks $i'$ ahead of $j$. 
    By \eqref{eq:bndForCut}, we have that
    $$
      \pi(i')\geq i'-\frac{1}{20}\bigl(|i-s|+|i'-s|\bigr)
      \qquad\text{and}\qquad
      \pi(j)\leq j+\frac{1}{20}\bigl(|i-s|+|j-s|\bigr),
    $$
    so it suffices to verify that
    \begin{equation}\label{eq:f}
      i'-\frac{1}{20}\bigl(|i-s|+|i'-s|\bigr)>j+\frac{1}{20}\bigl(|i-s|+|j-s|\bigr).
    \end{equation}
    The left side of this inequality is increasing in $i'$, and the right side is increasing in $j$.
    Thus by \eqref{eq:primeChoice} it suffices to verify it with $\tfrac23 i+\tfrac13 s$ in place of $i'$, and by \eqref{eq:jBound} it suffices to verify it with $\tfrac13 i+\tfrac23 s$ in place of $j$, i.e.\ \eqref{eq:f} is implied by
    \begin{equation}\label{eq:g}
      (\tfrac23 i+\tfrac13 s)-\frac{|i-s|}{20}-\frac{|(\tfrac23 i+\tfrac13 s)-s|}{20}>(\tfrac13 i+\tfrac23 s)+\frac{|i-s|}{20}+\frac{|(\tfrac13 i+\tfrac23 s)-s|}{20}.
    \end{equation}
    The expressions inside the absolute values in \eqref{eq:g} are all positive, and upon expansion the inequality reduces to $\tfrac{11}{60}i>\tfrac{11}{60}s$. Thus $\pi(i')>\pi(j)$.

    By stability of the matching, it follows that $\mathcal P\bigl((i',\male)\bigr)$ ranks $j'$ ahead of $i$. Thus, by another application of \eqref{eq:bndForCut}, we have that
    $$
      j'+\frac{1}{20}\bigl(|i'-s|+|j'-s|\bigr)>i-\frac{1}{20}\bigl(|i'-s|+|i-s|\bigr).
    $$
    Since the left side is increasing in $j'$ and $j'<\tfrac23i+\tfrac13s$, this inequality implies that
    $$
      \tfrac23i+\tfrac13s+\frac{1}{20}\bigl(|i'-s|+|(\tfrac23i+\tfrac13s)-s|\bigr)>i-\frac{1}{20}\bigl(|i'-s|+|i-s|\bigr).
    $$
    All quantities in the absolute values are seen to be positive, and after some algebra this inequality reduces to $i'-s>\tfrac52(i-s)$. In particular, $i'>i$. Moreover, since $i-s>\tfrac32(j'-s)$ by \eqref{eq:primeChoice}, we have that
    \begin{equation}\label{eq:keyIneq}
      i'-s>\frac{5}{2}\cdot \frac{3}{2}(j'-s)>3(j'-s),
    \end{equation}
    implying that $(i',j')\in \mathcal S$. This contradicts maximality of the pair $(i,j)$, showing that the set $\mathcal S$ is in fact empty. Thus $s$ is a cutpoint of the stable matching $M$. 
    The result follows since $M$ was chosen arbitrary from the stable matchings of $\mathcal P$.
  \end{proof}

  Having succeeded in locating \textit{cutpoints}, we now strengthen the criterion in order to locate \textit{essential lattice cutpoints}. 
  Recall that the cutpoint bound \eqref{eq:bndForCut} holds if $|\pi(j)-j|$ is sufficiently small for all relevant $\pi$ and $j$. 
  One may regard $|\pi(j)-j|$ as a local measure of deviation from the identity. 
  However, this quantity is non-monotonic with respect to induced permutations: it may decrease or increase upon passage to an induced permutation. 
  We start by introducing quantities that do not suffer from this issue. 

  For integers $a\leq j\leq b$ and a permutation $\pi$ of $\bbb{a,b}$, we set
  \begin{equation}\label{eq:l++}
  \begin{aligned}
    L_{++}&=\#\{j < i\leq b\colon \pi(i)>\pi(j)\},\quad   
    L_{+-}=\#\{j < i\leq b\colon \pi(i)<\pi(j)\},\\
    L_{-+}&=\#\{a\leq i<j\colon \pi(i)>\pi(j)\},\quad
    L_{--}=\#\{a\leq i<j\colon \pi(i)<\pi(j)\}.
  \end{aligned} 
  \end{equation}
  In particular,
  \begin{equation}\label{eq:l++ineq}
    L_{++},L_{+-}\in\bbb{0,b-j}\qquad\text{and}\qquad L_{-+},L_{--}\in \bbb{0,j-a}.
  \end{equation}
  As appropriate, we write $L_{++}=L_{++}(\pi,j)=L_{++}^{\bbb{a,b}}(\pi,j)$ and so on. 
  We set
  \begin{equation}\label{eq:offset}
    \offset(\pi,j):=\max\bigl(L_{+-}(\pi,j),L_{-+}(\pi,j)\bigr).
  \end{equation}
  The quantity $\offset(\pi,j)$ serves as a monotonic replacement for $|\pi(j)-j|$.
  \begin{lemma}\label{lem:diffLehm}
    Fix a finite interval $J\subset \m Z$ and a permutation $\pi$ of $J$.
    \begin{enumerate}[(i)]
      \item For all $j\in J$, we have that $|\pi(j)-j|\leq \offset(\pi,j)$.
      \item For any subinterval $I\subseteq J$ and any $i\in I$, we have that $\offset(\pi_I,i)\leq \offset(\pi,i).$
      \item If $i,j\in J$ satisfy $i\geq j$ and $\pi(i)\leq \pi(j)$, then  \label{empty nest}
      \begin{equation*}%\label{eq:offsetN}
        \max\bigl(\offset(\pi,i),\offset(\pi,j)\bigr)\geq i-j.
      \end{equation*}    
    \end{enumerate}
  \end{lemma}
  \begin{proof}
    Adding the equations $L_{++}+L_{+-}=b-j$ and $L_{+-}+L_{--}=\pi(j)-a$ yields
    %By $L_{++}+L_{+-}=\#\{j<i\leq b\}=b-j$ and $$L_{+-}+L_{--}=\#\{a\leq i\leq b\colon \pi(i)<\pi(j)\}=\pi(j)-a.$$ 
    %Adding these equations yields
    \begin{equation}\label{eq:1L}
      L_{++}+2L_{+-}+L_{--}=\pi(j)-j+b-a,
    \end{equation}
    whereas
    \begin{equation}\label{eq:2L}
      L_{++}+L_{+-}+L_{-+}+L_{--}=\#\{a\leq i\leq b\colon \pi(i)\not=\pi(j)\}=b-a.
    \end{equation}
    Subtracting \eqref{eq:2L} from \eqref{eq:1L} yields
    $$
      \pi(j)-j=L_{+-}(\pi,j)-L_{-+}(\pi,j),
    $$
    which implies $(i)$.

    Property $(ii)$ is immediate from the definitions. 

    To establish $(iii)$, apply $(i)$ twice to obtain that
    $$
      \max\bigl(|\pi(i)-i|,|\pi(j)-j|\bigr)\leq \max\bigl(\offset(\pi,i),\offset(\pi,j)\bigr).
    $$
    The bound now follows since $|\pi(j)-j-\pi(i)+i|\leq \max\bigl(|\pi(i)-i|,|\pi(j)-j|\bigr)$ and 
    \begin{equation*}
      |\pi(j)-j-\pi(i)+i|=\pi(j)-\pi(i) + i-j\geq i-j.\qedhere
    \end{equation*}
  \end{proof}
  
  The next lemma is the main result of this section.
  \begin{lemma}\label{lem:strongCutBound}
    Let $s\in\m Z+\tfrac12$ be a half-integer, let $I\subseteq \m Z$ be a (possibly infinite) interval, and let $\mathcal P$ be a preference structure on $\Omega_{I}$. 
    If for all $i,j\in I$, for all finite intervals $J$ with $\{i,j\}\subseteq J\subseteq I$, and for both \textup{$\malefemale\in \{\male,\female\}$},
    \begin{equation}\label{eq:strongCutBound}
      \offset\Bigl(\mathcal P_J\bigl((i,\textup{$\malefemale$})\bigr),j\Bigr)<\frac{|i-s|+|j-s|}{20},
    \end{equation}
    then $\lfloor s\rfloor$ is an essential lattice cutpoint of $\mathcal P$.
  \end{lemma}
  We say that $\mathcal P$ satisfies the \textbf{lattice cutpoint bound} at $s$ whenever \eqref{eq:strongCutBound} holds for all $i$, $j$, $J$, and $\malefemale$ as above. 
  By Lemma~\ref{lem:diffLehm} $(i)$ and $(ii)$, if a preference structure $\mathcal P$ on $\Omega_I$ satisfies the lattice cutpoint bound at $s$ then $\mathcal P_J$ satisfies the cutpoint bound at $s$ for every finite interval $J$ satisfying $s\in J\subseteq I$. 
  \begin{proof}[Proof of Lemma~\ref{lem:strongCutBound}]
    Fix $s$ and $\mathcal P$ satisfying the conditions of the lemma. 
    We must show that for all integers $i$ and $j$ with $i<s< j$, the induced preference structure $\mathcal P_{\bbb{i,j}}$ has a lattice cutpoint at $\lfloor s\rfloor$. 
    By Lemma~\ref{lem:diffLehm} $(i)$, \eqref{eq:strongCutBound}, and Lemma~\ref{lem:cutBound}, every stable matching of $\mathcal P_{\bbb{i,j}}$ has a cutpoint at $\lfloor s\rfloor$. 
    To verify that $k=\lfloor s \rfloor$ is a \textit{lattice} cutpoint, it remains to show that the union of any stable matching of $\mathcal P_{\bbb{i,k}}$ with any stable matching of $\mathcal P_{\bbl{k,j}}$ is stable under $\mathcal P_{\bbb{i,j}}$.

    Fix any two such stable matchings and suppose to the contrary that the union $M$ is not stable with respect to $\mathcal P_{\bbb{i,j}}$. 
    Then there exist integers $a,b\in\bbb{i,j}$ such that $(a,\female)$ and $(b,\male)$ each prefer the other (under $\mathcal P_{\bbb{i,j}}$) to their respective partners under $M$. 
%    Let $s=k+\tfrac12$. 
    By stability of the restricted matchings, either $a>s>b$ or $a<s<b$. 
    Assume, without loss of generality, that 
    \begin{equation}\label{eq:isj}
      a>s>b.
    \end{equation}
    Let $(a',\female)$ and $(b',\male)$ be the partners of $(b,\male)$ and $(a,\female)$ under $M$, respectively. 
    Since $M$ is a union of matchings of $\Omega_{\bbb{i,k}}$ and $\Omega_{\bbl{k,j}}$, it follows that $b'>s>a'$. 

    By definition of $a$, $b$, and induced preference structures, the permutation $\pi=\mathcal P\bigl((a,\female)\bigr)$ has $\pi(b)>\pi(b')$.  
    By Lemma~\ref{lem:diffLehm}, the induced preference structure $\mathcal P_{\bbl{k,j}}$ satisfies the cutpoint bound \eqref{eq:bndForCut} at $s$. 
    Hence by \eqref{eq:keyIneq} in the proof of Lemma~\ref{lem:cutBound}, the match $\{(a,\female),(b',\male)\}$ satisfies 
    $$
      a-s\leq 3(b'-s),
    $$
    for otherwise it would be contained in the set shown to be empty in the proof.
    That is, 
    \begin{equation}\label{eq:jprimelower}
      b'\geq \tfrac13 a+\tfrac23 s.
    \end{equation}

    Combining the cutpoint bound with Lemma \ref{lem:diffLehm} \eqref{empty nest} and $\pi(b)>\pi(b')$ yields 
    \begin{equation}\label{eq:contj}
      b+\frac{1}{20}\bigl(|a-s|+|b-s|\bigr)>b'-\frac{1}{20}\bigl(|a-s|+|b'-s|\bigr).
    \end{equation}
    The left and right sides of \eqref{eq:contj} are increasing in $b$ and $b'$, respectively. 
    By \eqref{eq:isj}, replacing each occurrence of $b$ with $s$ on the left side increases its value. 
    By \eqref{eq:jprimelower}, replacing each occurrence of $b'$ with $\tfrac13 a+\tfrac23 s$ on the right decreases its value. 
    Thus \eqref{eq:contj} implies that
    $$
      s+\frac{|a-s|}{20}>(\tfrac13 a+\tfrac23 s)-\frac{1}{20}\bigl(|a-s|+|(\tfrac13 a+\tfrac23 s)-s|\bigr).
    $$
    The expressions inside the absolute values are all positive, and upon expansion the inequality simplifies to $\tfrac{13}{60}s>\tfrac{13}{60}a$, contradicting \eqref{eq:isj}.
  \end{proof}

  %%%%%%%%%%%%%%%
  %%%%%%%%%%%%%%%
  %% Positive  %%
  %%%%%%%%%%%%%%%
  %%%%%%%%%%%%%%%

  \section{Two positive probabilities}\label{sec:positive}

  In this section we show that two key probabilities are positive for all $q\in (0,1)$:
  \begin{itemize}
    \item the probability $\rho_q$ that $0$ is an essential lattice cutpoint of $\MalPref_{q,\m Z}$, and
    \item the probability that $\LocPref(\MalPref_{q,\m Z},0)$ has multiple stable matchings.
  \end{itemize}

  \begin{prop}\label{prop:rhoP}
    For all $q\in (0,1)$ we have that $\rho_q>0$, and $\rho_q\to 1$ as $q\to 0$.
  \end{prop}

  \begin{prop}\label{prop:grBound}
    For all $q\in (0,1)$ we have that $\gr(\MalPref_{q,\m Z},0)>0$ with positive probability.
  \end{prop}

  These results are established by careful analysis of the lattice cutpoint bound \eqref{eq:strongCutBound}. 
  % Only 2 pages ago, no need to remind yet again.
  % Recall that this bound holds for a preference structure $\mathcal P$ on $\Omega_I$ at location $s$ if for all $i,j\in I$ and both \textup{$\malefemale\in \{\male,\female\}$}, 
  % $$
  %   \offset\bigl(\mathcal P\bigl((i,\textup{$\malefemale$})\bigr),j\bigr)<\frac{|i-s|+|j-s|}{20}.
  % $$
  % To show that this bound occurs with positive probability for $\mathcal P=\MalPref_{q,\m Z}$ and for an arbitrary half-integer $s\in\m Z+\tfrac12$, we introduce a cutoff parameter $N$. 
  % We argue that there is an (explicit) positive probability that $$\offset\bigl(\mathcal P\bigl((i,\malefemale)\bigr),j\bigr)=0$$ whenever $|i-s|+|j-s|<N$, and conditioned on this event, there is a tiny probability that the bound is violated when $|i-s|+|j-s|\geq N$ (provided $N$ is taken sufficiently large). 
%
%
%
%  satisfies the latt
  We begin by estimating some fundamental Mallows probabilities.

  %We prove this by describing an explicit preference structure with $2$ stable matchings, and showing that there is a positive probability that this occurs as $\Nbhd(\MalPref_{q,\m Z},0)$.

  %Before proceeding we record the following bounds.
  \begin{lemma}\label{lem:offMalBound}
    Let $I\subset \m Z$ be a finite interval and let $\pi$ be a random permutation of $I$ with law $\Mal_{q,I}$. 
    For all real $t\geq 0$ and all $j\in I$,
    \begin{equation}\label{eq:offqt}
      \mathbb P\bigl(\offset(\pi,j)\geq t\bigr)\leq 2q^t.
    \end{equation}
    Furthermore for all $i,j\in I$ with $i\geq j$,
    \begin{equation}\label{eq:offqt2}
      \mathbb P\bigl(\pi(i)\leq \pi(j)\bigr)\leq 4q^{i-j}.
    \end{equation}
  \end{lemma}
  Recall the definitions of $L_{+-}(\pi,j)$ and $L_{-+}(\pi,j)$ in \eqref{eq:l++}. 
  Both count inversions involving $j$. 
  Also recall that $\offset(\pi,j)$ is the maximum of these two quantities \eqref{eq:offset}.
  \begin{proof}
    Inequality \eqref{eq:offqt2} follows from \eqref{eq:offqt}. Indeed, by Lemma~\ref{lem:diffLehm} $(iii)$, since $i\geq j$ we have that
    $$
      \mathbb P\bigl(\pi(i)\leq \pi(j)\bigr)\leq \mathbb P\bigl(\offset(\pi,i)\geq i-j\bigr)+\mathbb P\bigl(\offset(\pi,j)\geq i-j\bigr).
    $$

    % Observe that \eqref{eq:offqt2} follows from \eqref{eq:offqt} with $t=i-j$, when combined with \eqref{eq:offsetN} and the union bound. 

    We will deduce \eqref{eq:offqt} from the stronger inequalities
    \begin{equation}\label{eq:+-geom}
      \mathbb P\bigl(L_{+-}(\pi,j)\geq k\bigr)\leq q^k
    \end{equation}
    and
    \begin{equation}\label{eq:-+geom}
      \mathbb P\bigl(L_{-+}(\pi,j)\geq k\bigr)\leq q^k,
    \end{equation}
    where $k\geq 0$ is an integer.

    We prove \eqref{eq:+-geom}. The proof of \eqref{eq:-+geom} is similar and we leave it to the reader. 
    Consider the map from permutations $\sigma$ of $I$ to sequences given by
    $
      \sigma\mapsto \bigl(L_{+-}(\sigma,\ell)\bigr)_{\ell\in I}.
    $
    It is straightforward to verify by induction on $\# I$ that this map is injective. 
    Furthermore by \eqref{eq:l++ineq}, the image is contained in the set of sequences $(s_\ell)_{\ell\in I}$ satisfying $0\leq s_\ell\leq \max I-\ell$ for all $\ell\in I$. 
    There are $(\# I)!$ such sequences. 
    Since this equals the number of permutations, it follows that the map is a bijection onto this set of sequences.

    Consider the pushforward of $\Mal_{q,I}$ under this map. 
    Since $$\inv(\sigma)=\sum_{\ell\in I}L_{+-}(\sigma,\ell),$$ it follows that if $\pi$ has law $\Mal_{q,I}$ then $\bigl(L_{+-}(\pi,\ell)\bigr)_{\ell\in I}$ is a sequence of independent truncated geometric random variables. 
    In particular, 
    $$
      \mathbb P\bigl(L_{+-}(\pi,j)\geq k\bigr)=\frac{q^k+q^{k+1}+\cdots+q^{\max I-j}}{1+q+\cdots+q^{\max I-j}}\leq q^k,
    $$
    establishing \eqref{eq:+-geom}. 
    Replacing the bounds for $L_{+-}$ with those for $L_{-+}$ in the argument above yields \eqref{eq:-+geom}. 
    Lastly, \eqref{eq:offqt} follows from \eqref{eq:+-geom} and \eqref{eq:-+geom} by the union bound.
  \end{proof}

  For finite intervals of integers $I\subseteq J$, denote the set of elements of $J$ that are smaller than $\min I$ by $(J\setminus I)^-$, and denote the elements larger than $\max I$ by $(J\setminus I)^+$. 
  Note that $J=(J\setminus I)^-\cup I\cup (J\setminus I)^+$. 
  For a permutation $\tau$ of $I$, denote by $\Res(\tau,I,J)$ the set of permutations of $J$ which fix the sets $(J\setminus I)^{\pm}$ and whose restriction to $I$ is $\tau$. 
  For each $\tau$, there is an evident bijection from $\Res(\tau,I,J)$ to pairs consisting of a permutation of $(J\setminus I)^-$ and a permutation of $(J\setminus I)^+$. 
  Define $\Fix(I,J):=\Res(\id,I,J)$, where $\id$ is the identity permutation fixing every element of $I$.

  We estimate the probability that a Mallows-distributed permutation lies in the set $\Res(\tau,I,J)$. 
  In fact, we show that it is bounded below by a quantity exponentially small in $\# I$ and independent of $J$. 
  During this calculation, the function
  $$
    \phi(q):=\prod_{k=1}^{\infty}(1-q^k)
  $$
  will appear. 
  Note that $\phi(q)>0$ for all $q\in (0,1)$, since $\sum_{k\geq 1}q^k$ is finite. 
  \begin{lemma}\label{lem:fix}
    Fix finite intervals of integers $I\subseteq J$, fix a permutation $\tau$ of $I$, and fix $q\in [0,1)$. 
    Let $\pi$ be a random permutation of $J$ with law $\Mal_{q,J}$. 
    Conditional on the event that $\pi$ fixes each of $(J\setminus I)^-$, $I$, and $(J\setminus I)^+$, the restrictions of $\pi$ to each of $(J\setminus I)^-$, $I$, and $(J\setminus I)^+$ are independent, $\Mal_q$-distributed  permutations. Furthermore,
    \begin{equation}\label{eq:presB}
      \mathbb P\bigl(\pi\in \Res(\tau,I,J)\bigr)\geq q^{\inv(\tau)}(1-q)^{\# I}\phi(q).
    \end{equation}
  \end{lemma}

  In the proof we will use the following well-known fact (Cor.~1.3.13, p.\ 36 in \cite{stanley1997enumerative2ed}):
  \begin{equation}\label{eq:malPart}
    \sum_{\sigma\in S_n}q^{\inv(\sigma)}=\prod_{k=1}^n(1+q+\cdots+q^{k-1})=(1-q)^{-n}\prod_{k=1}^n(1-q^k).
  \end{equation}
  This is straightforward to verify by induction on $n$. Set $Z(q,n):=\sum_{\sigma\in S_n}q^{\inv(\sigma)}$.
  \begin{proof}[Proof of Lemma~\ref{lem:fix}]
    We start by proving conditional independence. 
    If $\sigma$ is a permutation of $J$ fixing the sets $(J\setminus I)^-$, $I$, and $(J\setminus I)^+$, then the restriction of $\sigma$ to each of these sets is equal to the induced permutation, and
    $$
      \inv(\sigma)=\inv\bigl(\sigma_{(J\setminus I)^{-}}\bigr)+\inv\bigl(\sigma_I\bigr)+\inv\bigl(\sigma_{(J\setminus I)^{+}}\bigr).
    $$
    From the form of the Mallows measure, it follows that the restrictions of $\sigma$ to each of $(J\setminus I)^-$, $I$, and $(J\setminus I)^+$ are independent Mallows-distributed permutations of the respective intervals. 

    Further conditioning $\pi_I$ to equal $\tau$ implies that, conditional on $\pi\in \Res(\tau,I,J)$, the random permutations $\pi_{(J\setminus I)^-}$ and $\pi_{(J\setminus I)^+}$ are conditionally independent and Mallows-distributed. Therefore
    \begin{align*}
      \mathbb P\bigl(\pi\in\Res(\tau,I,J)\bigr)&=q^{\inv(\tau)}\frac{Z\bigl(q,\#(J\setminus I)^-\bigr)\cdot Z\bigl(q,\#(J\setminus I)^+\bigr)}{Z\bigl(q,\#J\bigr)}\\
      &=q^{\inv(\tau)}(1-q)^{\# I}\frac{\prod_{i=1}^a(1-q^i)\prod_{i=1}^b(1-q^i)}{\prod_{i=1}^c(1-q^i)},
    \end{align*}
    where $a=\#(J\setminus I)^-$ and $b=\#(J\setminus I)^+$ and $c=\# J$. 
    The claimed lower bound on $\mathbb P\bigl(\pi\in\Res(\tau, I, J)\bigr)$ now follows from 
    \begin{equation*}
      \frac{\prod_{i=1}^a(1-q^i)\prod_{i=1}^b(1-q^i)}{\prod_{i=1}^c(1-q^i)}=
      \frac{\prod_{i=1}^a(1-q^i)}{\prod_{i=b+1}^c(1-q^i)}\geq \prod_{i=1}^a(1-q^i)\geq \phi(q).\qedhere
    \end{equation*}
  \end{proof}

  Having developed the requisite bounds for probabilities involving permutations, we now turn to bounds for entire preference structures. 
  Our aim is to show that the lattice cutpoint bound occurs with positive probability when $\mathcal P=\MalPref_{q,\m Z}$. 
  For this, we introduce a cutoff parameter $N$, show that there is a positive (but tiny) probability that
  $$
    \offset\bigl(\mathcal P\bigl((i,\malefemale)\bigr),j\bigr)=0
  $$
  whenever $|i-s|+|j-s|<N$, and show that there is a large conditional probability given this event that the bound is satisfied for all $i,j$ with $|i-s|+|j-s|\geq N$. 

  Let $N$ be a positive integer. Let $I$ be a finite interval of integers, and let $s\in\m Z+\tfrac12$ be a half-integer. Denote by $\mathscr B_I(s,N)$ the set of preference structures $\mathcal P$ on $\Omega_I$ such that, for all $i\in I$ with $|i-s|<N$ and for both $\malefemale\in \{\male,\female\}$, 
  $$
    \mathcal P\bigl((i,\malefemale)\bigr)\in \Fix\bigl(\{j\in I\colon |j-s|+|i-s|<N\},I\bigr).
  $$
  Observe that if $\mathcal P\in \mathscr B_I(s, N)$, then $\mathcal P$ satisfies the lattice cutpoint bound \eqref{eq:strongCutBound} for all $i,j\in I$ with $|i-s|+|j-s|<N$, since $\offset(\mathcal P\bigl((i,\malefemale)\bigr),j)=0$ for all such $i$ and $j$.

  We now show that $\rho_q$, the probability that $0$ is an essential lattice cutpoint of $\MalPref_{q,\m Z}$, is positive for all $q\in (0,1)$. 
  \begin{lemma}\label{lem:rhoP}
    For all $q\in (0,1)$, for all integers $N\geq 1$, and for all half-integers $s\in\m Z+\tfrac12$, the probability that $\MalPref_{q,\m Z}$ satisfies the lattice cutpoint bound \eqref{eq:strongCutBound} at $s$ is at least
    $$
      (1-q)^{8N^2}\phi(q)^{4N}\left(1-\Bigl(\frac{4q^{N/40}}{1-q^{1/20}}\Bigr)^2\right).
    $$
    In particular, this expression is a lower bound for $\rho_q$.
  \end{lemma}

  \begin{proof}
    Fix $q$, $N$, and let $\mathcal P=\MalPref_{q,\m Z}$. For all integers $n\geq 1$, let $\mathscr A_{\bbb{-n,n}}$ be the event that $\mathcal P_{\bbb{-n,n}}$ satisfies the lattice cutpoint bound \eqref{eq:strongCutBound} at $s=\tfrac12$. 
    By Lemma~\ref{lem:diffLehm} $(ii)$, the lattice cutpoint bound is inherited by induced permutations. 
    Thus the events $\mathscr A_{\bbb{-n,n}}$ are decreasing, and if they all occur then by Lemma~\ref{lem:strongCutBound}, the preference structure $\mathcal P$ has an essential lattice cutpoint at $\lfloor s\rfloor$. 
    Consequently,
    \begin{equation}\label{eq:rholiminf}
      \rho_q\geq \inf_{n\geq 1}\mathbb P\bigl(\mathscr A_{\bbb{-n,n}}\bigr).
    \end{equation}
    We proceed to find a uniform lower bound for $\mathbb P\bigl(\mathscr A_{\bbb{-n,n}}\bigr)$. 

    The preferences $\mathcal P\bigl((i,\malefemale)\bigr)$ for $(i,\malefemale)\in \Omega_{\m Z}$ are independent. 
    Therefore 
    \begin{equation}\label{eq:pnnb}
      \mathbb P\bigl(\mathcal P_{\bbb{-n,n}}\in \mathscr B_{\bbb{-n,n}}(s,N)\bigr)
      =
      \prod_{%
        \substack{%
          i\in A_n,\\
          \malefemale\in\{\male,\female\}
        }%
      }\mathbb P\Bigl(\mathcal P_{\bbb{-n,n}}\bigl((i,\malefemale)\bigr)\in \Fix\bigl(I_{n,i},\bbb{-n,n}\bigr)\Bigr),%\\
    \end{equation}
    where  $A_n=\{i\in \bbb{-n,n}\colon |i-s|<N\}$, $I_{n,i}=\{j\in \bbb{-n,n}\colon |j-s|+|i-s|<N\}$, and we take $n$ sufficiently large such that $A_n$ is non-empty. 
    By Lemma~\ref{lem:fix},
    \begin{equation}\label{eq:pfixq}
      \mathbb P\Bigl(\mathcal P_{\bbb{-n,n}}\bigl((i,\malefemale)\bigr)\in \Fix\bigl(I_{n,i},\bbb{-n,n}\bigr)\Bigr)\geq (1-q)^{\#I_{n,i}}\phi(q).
    \end{equation}
    Substitute \eqref{eq:pfixq} into \eqref{eq:pnnb} and use $\# A_n\leq 2N$ and $\sum_{i\in A_n}\# I_{n,i}\leq 4N^2$ to obtain that
    \begin{align}
      \mathbb P\bigl(\mathcal P_{\bbb{-n,n}}\in \mathscr B_{\bbb{-n,n}}(s,N)\bigr)&\geq \Bigl(\prod_{i\in A_n}(1-q)^{\# I_{n,i}}\phi(q)\Bigr)^2\nonumber\\
      &\geq (1-q)^{8N^2}\phi(q)^{4N}.\label{eq:BsNbnd}
    \end{align}

    The conditional probability that $\mathscr A_{\bbb{-n,n}}$ does \textit{not} occur given that $\mathcal P_{\bbb{-n,n}}$ belongs to $\mathscr B_{\bbb{-n,n}}(s,N)$ is bounded from above by
    \begin{equation}\label{eq:rbndij}
      \sum_{\malefemale\in\{\male,\female\}}\sum_{\substack{%
        i,j\in\bbb{-n,n}\colon\\
        |i-s|+|j-s|\geq N
      }}\mathbb P\Bigl(\offset\bigl(\mathcal P_{\bbb{-n,n}}((i,\malefemale)),j\bigr)\geq \frac{|i-s|+|j-s|}{20}\Bigr).
    \end{equation}
    By Lemma~\ref{lem:fix}, the conditional law of $\mathcal P_{\bbb{-n,n}}$ given that $\mathcal P_{\bbb{-n,n}}\in \mathscr B_{\bbb{-n,n}}(s,N)$ is such that for all $(i,\malefemale)\in\Omega_{\bbb{-n,n}}$ the restrictions of $\mathcal P_{\bbb{-n,n}}\bigl((i,\malefemale)\bigr)$ to $(\bbb{-n,n}\setminus I_n)^{\pm}$ are conditionally Mallows-distributed.    
    Thus \eqref{eq:offqt} in Lemma~\ref{lem:offMalBound} applies, so \eqref{eq:rbndij} is bounded from above by
    $$
      2\sum_{\substack{%
        i,j\in\bbb{-n,n}\colon\\
        |i-s|+|j-s|\geq N
      }}2q^{\frac{|i-s|+|j-s|}{20}}\leq 4\Bigl(\frac{2q^{N/40}}{1-q^{1/20}}\Bigr)^2,
    $$
    and therefore
    $$
      \mathbb P\Bigl(\mathscr A_{\bbb{-n,n}}\Bigm\vert \mathcal P_{\bbb{-n,n}}\in\mathscr B_{\bbb{-n,n}}(s,N)\Bigr)\geq 1-\Bigl(\frac{4q^{N/40}}{1-q^{1/20}}\Bigr)^2.
    $$
    The result follows by combining this with \eqref{eq:BsNbnd} and substituting into \eqref{eq:rholiminf}.
  \end{proof}
  \begin{proof}[Proof of Proposition~\ref{prop:rhoP}]
    That $\rho_q>0$ for all $q\in (0,1)$ follows by taking $N=N_q$ sufficiently large in Lemma~\ref{lem:rhoP}. 
    Taking $N=1$ and $q\to 0$ in the lemma yields that $\lim_{q\to 0}\rho_q=1$. 
  \end{proof}

  Having established the first of the two positive probabilities promised in this section, we turn to the second. 
  Our aim is to show that there is a positive probability that the random  preference structure $\LocPref(\MalPref_{q,\m Z},0)$ induced by a neighborhood of the origin has multiple stable matchings. 
  By the celebrated Gale-Shapley algorithm \cite{gale1962college}, $\LocPref(\MalPref_{q,\m Z},0)$ always has at least one stable matching. 
  That is,
  \begin{equation}\label{eq:grP}
    \gr(\MalPref_{q,\m Z},0)\geq 0.    
  \end{equation}

  % Note that $\gr(\mathcal P,0)\geq 0$ since every finite preference structure has at least one stable matching, by the celebrated Gale-Shapley algorithm \cite{gale1962college}.

  Consider the unique preference structure $\mathcal Q$ on $\Omega_{\m Z}$ satisfying the following properties:
  \begin{itemize}
    \item individuals in $\Omega_{\m Z}\setminus \Omega_2$ rank everyone in order, and
    \item individuals in $\Omega_2$ rank individuals in $\Omega_{\m Z}\setminus \Omega_2$ in order, and
    \item the induced preference structure $\mathcal T$ on $\Omega_2$ is given by
    \begin{align*}
      \mathcal T\bigl((1,\female)\bigr)=12,\qquad &\mathcal T\bigl((2,\female)\bigr)\hspace{.25em}=21,\\
      \mathcal T\bigl((1,\male)\bigr)=21,\qquad &\mathcal T\bigl((2,\male)\bigr)=12.
    \end{align*}
  \end{itemize}
  In other words, preferences in $\mathcal Q$ are in order except that  $(2,\female)$ prefers $(1,\male)$ to $(2,\male)$ and $(1,\male)$ prefers $(1,\female)$ to $(2,\female)$. 
  Observe that for all $m\geq 2$, the induced preference structure $\mathcal Q_{\bbb{-m,m}}$ has exactly $2$ stable matchings: the in-order matching $\bigcup_{i\in\bbb{-m,m}}\bigl\{(i,\female),(i,\male)\bigr\}$ and
  $$
    \bigl\{(1,\female),(2,\male)\bigr\}\cup \bigl\{(2,\female),(1,\male)\bigr\}\cup\bigcup_{i\in\bbb{-m,m}\setminus\{1,2\}}\bigl\{(i,\female),(i,\male)\bigr\}.
  $$

  The proof of Proposition~\ref{prop:grBound} is quite similar to the proof of Proposition~\ref{prop:rhoP}, since both amount to obtaining lower bounds on the probability that the lattice cutpoint bound (or a close analogue thereof) is satisfied. 
  In light of this, we present certain steps of the proof of Proposition~\ref{prop:grBound} in a complete yet concise manner, with the understanding that a lengthier exposition of the steps in question appears in the proof of Proposition~\ref{prop:rhoP}.

  We recall the following notation for a preference structure $\mathcal P$ on $\Omega_{\m Z}$. \begin{itemize}
  \item $\Nbhd(\mathcal P,0)$ is the interval whose endpoints are the two essential lattice cutpoints surrounding $0$ (inclusive for the right endpoint and exclusive for the left). 
  \item $\LocPref(\mathcal P,0)$ is the preference structure induced by $\Nbhd(\mathcal P,0)$.
  \item $\gr(\mathcal P,0)$ is defined via
  $$
    \gr(\mathcal P,0)=\frac{\log\#\!\stab\LocPref(\mathcal P,0)}{\# \Nbhd(\mathcal P,0)}
  $$
  if $\Nbhd(\mathcal P,0)$ is finite and $\gr(\mathcal P,0)=0$ otherwise.
  \end{itemize}

  For an integer $i\in\m Z$ and a non-empty set $S\subseteq \m Z$, the \textbf{distance} from $i$ to $S$ is 
  $$
    d(i,S):=\min_{j\in S}|i-j|.
  $$
  \begin{proof}[Proof of Proposition~\ref{prop:grBound}]
    Fix $q\in (0,1)$ and let $\mathcal P=\MalPref_{q,\m Z}$ be a random Mallows-distributed preference structure on $\Omega_{\m Z}$. 
    By the previous proposition $\rho>0$, and thus by Lemma~\ref{lem:latDense} the set of essential lattice cutpoints has positive density. 
    Let $I$ be the a.s.\ finite interval whose endpoints are the essential lattice cutpoints of $\mathcal P$ surrounding the origin, exclusive on the left and inclusive on the right. 
    We show that $\gr(\mathcal P,0)>0$ with positive probability for all $q\in (0,1)$, which is equivalent showing that $\mathcal P_{I}$ has multiple stable matchings with positive probability. 

    Choose $m\geq 2$ sufficiently large such that the induced matching $\mathcal Q_{\bbb{-m,m}}$ satisfies the cutpoint bound \eqref{eq:bndForCut} at $\pm m$ (for the sake of concreteness, $m=50$ suffices). 
    Let $\mathscr E$ be the event that $\mathcal P_{\bbb{-m,m}}=\mathcal Q_{\bbb{-m,m}}$ and that for every finite interval of integers $J$ containing $\bbb{-m,m}$, the induced preference structure $\mathcal P_J$ satisfies the cutpoint bound at $\pm m$ as well.

    We claim that $\mathbb P(\mathscr E)>0$. 
    More precisely, we show that for all integers $N\geq 1$,
    \begin{equation}\label{eq:Ebnd}
      \mathbb P(\mathscr E)\geq q^2(1-q)^{8(N+m)^2}\phi(q)^{4(N+m)}\left(1-\Bigl(\frac{4q^{N/40}}{1-q^{1/20}}\Bigr)^2\right),
    \end{equation}
    by a calculation similar to the proof of Lemma~\ref{lem:rhoP}. This is positive for $N=N_q$ large enough.
    
    For all integers $n\geq m$, let $\mathscr A_{\bbb{-n,n}}'$ be the event that $\mathcal P_{\bbb{-n,n}}$ satisfies the lattice cutpoint bound \eqref{eq:strongCutBound} at $\pm m$ and that $\mathcal P_{\bbb{-m,m}}=\mathcal Q_{\bbb{-m,m}}$ (c.f.\ $\mathscr A_{\bbb{-n,n}}$ in the proof of Lemma~\ref{lem:rhoP}.)

    %Fix an integer $N\geq 1$, which will serve as a cutoff parameter. 
    For all integers $n\geq m$, let $A_n'=\bigl\{i\in \bbb{-n,n}\colon d(i,\bbb{-m,m})<N\bigr\}$. 
    For all $i\in A_n'$, let
    $$
      I_{n,i}'=\bigl\{j\in \bbb{-n,n}\colon d(i,\bbb{-m,m})+d(j,\bbb{-m,m})<N\bigr\}
    $$
    (c.f.\ the sets $A_n$ and $I_{n,i}$ defined in the proof of Lemma~\ref{lem:rhoP}).

    Let $\mathscr B_{\bbb{-n,n}}'$ denote the set of preference structures $\mathcal R$ on $\Omega_{\bbb{-n,n}}$ such that, for all $i\in A_n$ and for both $\malefemale\in\{\male,\female\}$,
    $$
      \mathcal R\bigl((i,\malefemale)\bigr)\in\Res\Bigl(\mathcal Q_{\bbb{-n,n}}\bigl((i,\malefemale)\bigr),I_{n,i},\bbb{-n,n}\Bigr).
    $$
    By independence of the preferences $\mathcal P\bigl((i,\malefemale)\bigr)$ and Lemma~\ref{lem:fix},
    $$
      \mathbb P\bigl(\mathcal P_{\bbb{-n,n}}\in \mathscr B_{\bbb{-n,n}}'\bigr)\geq \Bigl(q\prod_{%
        \substack{%
          i\in A_n'%
        }} (1-q)^{\# I_{n,i}'}\phi(q)
        \Bigr)^2.
    $$
    (There is an extra factor of $q$, when compared with the calculation in Lemma~\ref{lem:rhoP}, since the restricted preference lists are taken to have a total of one inversion for each gender.) 
    Since $\# A_n'\leq 2(N+m)$ and $\sum_{i\in A_n'}\# I_{n,i}'\leq 4(N+m)^2$,
    \begin{equation}\label{eq:bnn'}
      \mathbb P\bigl(\mathcal P_{\bbb{-n,n}}\in \mathscr B_{\bbb{-n,n}}'\bigr)\geq q^2(1-q)^{8(N+m)^2}\phi(q)^{4(N+m)}.
    \end{equation}

    As in the proof of Lemma~\ref{lem:rhoP}, it follows from Lemma~\ref{lem:fix} that the conditional probability that $\mathscr A_{\bbb{-n,n}}'$ does \textit{not} occur given that $\mathcal P_{\bbb{-n,n}}$ belongs to $\mathscr B_{\bbb{-n,n}}'$ is bounded from above by
    $$
      2\sum_{%
        \substack{%
          i,j\in\bbb{-n,n}\colon\\
          d(i,\bbb{-m,m})+d(j,\bbb{-m,m})\geq N
        }} 2q^{\frac{d(i,\bbb{-m,m})+d(j,\bbb{-m,m})}{20}}\leq 4\Bigl(\frac{2q^{N/40}}{1-q^{1/20}}\Bigr)^2,
    $$
    and therefore 
    $$
      \mathbb P\Bigl(\mathscr A_{\bbb{-n,n}}'\Bigm\vert \mathcal P_{\bbb{-n,n}}\in \mathscr B_{\bbb{-n,n}}'\Bigr)\geq 1-4\Bigl(\frac{2q^{N/40}}{1-q^{1/20}}\Bigr)^2.
    $$
    Combining this with \eqref{eq:bnn'} yields \eqref{eq:Ebnd}. 
    Now taking $N=N_q$ sufficiently large yields that $\mathbb P(\mathscr E)>0$, as claimed.

    On $\mathscr E$, the lattice cutpoint bound holds at $\pm m$ and therefore every stable matching of $\mathcal P_I$ restricts to a stable (perfect) matching of $\mathcal P_{\bbb{-m,m}}$, yielding
    $$
      \#\!\stab \mathcal P_I\geq \#\!\stab \mathcal P_{\bbb{-m,m}}=\#\!\stab \mathcal Q_{\bbb{-m,m}}=2.
    $$
    Consequently,
    \begin{equation}\label{eq:2lwr}
      \mathbb P(\#\!\stab \mathcal P_I\geq 2)\geq \mathbb P(\mathscr E)>0.
      %2\leq\#\!\stab\mathcal Q_{\bbb{-n,n}}\leq \#\!\stab \mathcal P_I\qquad \text{on }E.
    \end{equation}

    Finally, we  deduce from \eqref{eq:2lwr} that $\gr(\mathcal P,0)>0$ with positive probability. 
    Indeed, suppose to the contrary that $\gr(\mathcal P,0)=0$ a.s.\ 
    Then by stationarity, $\gr(\mathcal P,i)=0$ a.s.\ for all $i\in \m Z$ as well, whereupon it holds with probability $1$ that
    $$
      \log\#\!\stab\mathcal P_I\leq 
      \# I\sum_{i\in\bbb{-n,n}}\gr(\mathcal P,i)=0.
    $$
    %implying that there is a.s.\ a unique stable matching of $\mathcal P_I$, 
    This contradicts \eqref{eq:2lwr}, and thus $\gr(\mathcal P,0)>0$ with positive probability.
  \end{proof}

  %%%%%%%%%%%%
  %%%%%%%%%%%%
  %%%%%%%%%%%%
  %%%%%%%%%%%%
  %%%%%%%%%%%%
  %%% main %%%
  %%%%%%%%%%%%
  %%%%%%%%%%%%
  %%%%%%%%%%%%
  %%%%%%%%%%%%

  \section{Proof of main theorem}\label{sec:expo}

    In this section we prove the main theorem, which states that for all $q\in (0,1)$, the number of stable matchings of $\MalPref_{q,n}$ grows exponentially in $n$, and the growth rate tends to $0$ as $q\to 0$. 
    Recall the quantity $\gr(\MalPref_{q,\m Z},0)$, previously introduced as a local approximation of the growth rate. 
    In Section \ref{sec:ergo} we used the ergodic theorem to show that the exponential growth rate of the number of stable matchings is given by
    $$
      \gr_q=\mathbb E\gr(\MalPref_{q,\m Z},0),
    $$
    for those values of $q\in (0,1)$ for which $\gr(\MalPref_{q,\m Z},0)$ is integrable. 
    Specifically, this follows by combining the results of Propositions~\ref{prop:main}, \ref{prop:rhoP}, and \ref{prop:grBound}. 
    To complete the proof of the main theorem, it remains to show that $\gr(\MalPref_{q,\m Z},0)$ is bounded for all $q\in (0,1)$, and that its mean tends to $0$ as $q\to 0$. 
    Throughout this section, functions of $q$ on $(0,1)$ tending to $0$ as $q\to 0$ are denoted by $o_q(1)$. 

%    \begin{prop}\label{prop:moment}
%      $\gr_q=\mathbb E\gr(\MalPref_{q,\m Z},0) \to0$ as $q\to 0$.
%    \end{prop}

%    We prove this by obtaining tail bounds on the distance from the origin to the closest essential lattice cutpoints. 
%    Since we already know that the set of essential lattice cutpoints has positive density by Lemma~\ref{lem:rhoP}, the only way for this distance to have heavy tails is if the essential lattice cutpoints are highly correlated. 

    Working with essential lattice cutpoints directly is challenging, so as in the previous section we instead consider the set of locations at which the lattice cutpoint bound \eqref{eq:strongCutBound} holds. 
    This is a subset of the set of essential lattice cutpoints, and along the way to proving Proposition~\ref{lem:rhoP} we showed that it, too, has  positive density. 
%    Here we establish that for fixed locations $s$ and $t$, the covariance between the event that the lattice cutpoint bound holds at $s$ and that it holds at $t$ decays exponentially in $|s-t|$. 
%    This is a straightforward (if technical) computation, since all relevant probabilities in the Mallows model decay exponentially in the separation distance. 
%    Summing this bound over all pairs of locations yields a variance bound, and we use the second moment method 
    Working with this set of witnesses for an essential lattice cutpoints
    to conclude the following.

    %For the remainder of this section, functions of $q$ defined on $(0,1)$ and tending to $0$ as $q\to 0$ are denoted by $o_q(1)$.
    \begin{lemma}\label{lem:offRbnd}
      Let $I\subseteq J$ be finite intervals of integers and let $\pi$ be a random permutation of $J$ with law $\Mal_{q,J}$, for $q\in [0,1)$. Then for all $j\in I$,
      $$
        \mathbb P\bigl(\offset(\pi,j)>\offset(\pi_I,j)\bigr)\leq \bigl(c_{1}+o_q(1)\bigr)q^{c_2d(j,\m Z\setminus I)},%\frac{8q^{d(j,\m Z\setminus I)}}{1-q}.
      $$
      where $c_1,c_2>0$ are absolute constants and $o_q(1)$ tends to $0$ as $q\to 0$.
    \end{lemma}
    \begin{proof}
      If $\offset(\pi,j)>\offset(\pi_I,j)$, then either $L_{+-}(\pi,j)>L_{+-}(\pi_I,j)$ or $L_{-+}(\pi,j)>L_{-+}(\pi_I,j)$. 
      Suppose that the former occurs. 
      By definition of $L_{+-}$, it follows that there is some $i\in J\setminus I$ with $i>j$ and $\pi(i)<\pi(j)$. 
      Thus by the union bound and Lemma~\ref{lem:diffLehm} $(iii)$,
      \begin{align}\label{eq:l+-f}
        \mathbb P\bigl(L_{+-}(\pi,j)>L_{+-}(\pi_I,j)\bigr)&\leq
        \sum_{%
          \substack{%
            i\in J\setminus I\colon\\
            i>j
          }}\mathbb P\bigl(\pi(i)<\pi(j)\bigr)\nonumber\\
          &\leq \sum_{%
          \substack{%
            i\in (J\setminus I)^+
          }}4q^{i-j}=\frac{4q^{\max I+1-j}}{1-q}.
      \end{align}
      By a similar argument,
      \begin{equation}\label{eq:l+-s}
        \mathbb P\bigl(L_{-+}(\pi,j)>L_{-+}(\pi_I,j)\bigr)\leq \frac{4q^{j-\min I-1}}{1-q}.
      \end{equation}
      The result now follows by \eqref{eq:l+-f}, \eqref{eq:l+-s}, and the union bound.
    \end{proof}

    Fix $q\in (0,1)$. 
    For a finite interval $I\subset \m Z$ and a preference structure $\mathcal P$ on $\Omega_I$, let $\Bnd(\mathcal P)\subseteq \m Z+\tfrac12$ denote the set of half-integers at which the lattice cutpoint bound \eqref{eq:strongCutBound} holds for $\mathcal P$. 
    If $\mathcal P$ is a preference structure on $\Omega_{\m Z}$, then by definition of the lattice cutpoint bound we have that
    $$
      \Bnd(\mathcal P)=\bigcap_{%
        \substack{%
          i,j\in\m Z\colon\\
          i\leq j
        }}\Bnd(\mathcal P_{\bbb{i,j}}).
    $$

    \begin{lemma}\label{lem:resDiffBnd}
      For all $q\in (0,1)$, all $s\in\m Z+\tfrac12$, and all finite intervals $I\subset \m Z$,
      $$
        \mathbb P\Bigl(s\in \Bnd\bigl((\MalPref_{q,\m Z})_I\bigr)\setminus \Bnd(\MalPref_{q,\m Z})\Bigr)\leq \# I\cdot \bigl(c_{1}+o_q(1)\bigr) \cdot q^{c_2d(s,\m Z\setminus I)},
      $$
      where $c_1,c_2>0$ are absolute constants and $o_q(1)$ tends to $0$ as $q\to 0$.
    \end{lemma}
    \begin{proof}
      Fix $q$, $I$, and let $\mathcal P=\MalPref_{q,\m Z}$. 
      If $s\not\in \Bnd(\mathcal P)$, then there exist $i,j\in \mathbb Z$ and $\malefemale\in\{\male,\female\}$ such that
      \begin{equation}\label{eq:viol}
        \offset\Bigl(\mathcal P\bigl((i,\malefemale\bigr),j\Bigr)\geq \frac{|i-s|+|j-s|}{20}.
      \end{equation}
      Let $\malefemale$, $i$, and $j$ be a triple satisfying \eqref{eq:viol}. Either
      \begin{enumerate}[$(i)$]
        \item $i\in \mathbb Z\setminus I$ or $|j-s|>d(s,\mathbb Z\setminus I)/2$, or
        %\item $|j-s|>d(s,\mathbb Z\setminus I)/2$, or
        \item $i\in I$ and $|j-s|\leq d(s,\mathbb Z\setminus I)/2$.
      \end{enumerate}

      By the union bound and \eqref{eq:offqt} from Lemma~\ref{lem:offMalBound}, we have that
      $$
        \mathbb P\bigl(\exists \ i\in \mathbb Z\setminus I,\ j\in\mathbb Z,\ \malefemale\in\{\male,\female\}\text{ satisfying }\eqref{eq:viol}\bigr)\leq 2\sum_{%
          \substack{%
            i\in\mathbb Z\setminus I\\
            j\in\mathbb Z}} 2q^{\frac{|i-s|+|j-s|}{20}}.
      $$
      Similarly, the probability that there exist $i\in \mathbb Z$, $j\in\mathbb Z$, and $\malefemale\in\{\male,\female\}$ satisfying $|j-s|>d(s,\mathbb Z\setminus I)/2$ and \eqref{eq:viol} is bounded from above by
      $$
        %\mathbb P\bigl(\exists \ i\in \mathbb Z,\ j\in\mathbb Z,\ \malefemale\in\{\male,\female\}\text{ satisfying }|j-s|>\tfrac{d(s,\mathbb Z\setminus I)}{2}\text{ and }\eqref{eq:viol}\bigr)\leq 
        2\sum_{%
          \substack{%
            i,j\in\mathbb Z\colon\\
            |j-s|>d(s,\mathbb Z\setminus I)/2}} 2q^{\frac{|i-s|+|j-s|}{20}}=\bigl(c_3+o_q(1)\bigr)q^{c_4d(s,\mathbb Z\setminus I)},
      $$
      for some absolute constants $c_3,c_4>0$. 
      Adding these two bounds together yields an upper bound on the probability that some $i$, $j$, and $\malefemale$ satisfy $(i)$ and \eqref{eq:viol}.

      Finally, we bound the probability that $s\in\Bnd(\mathcal P_I)$ and that both $(ii)$ and \eqref{eq:viol} are satisfied for some $\malefemale$, $i$, and $j$. 
      If all these conditions are satisfied, then 
      \begin{equation}\label{eq:viol2}
        \offset\Bigl(\mathcal P\bigl((i,\malefemale\bigr),j\Bigr)\geq \frac{|i-s|+|j-s|}{20}>\offset\Bigl(\mathcal P_I\bigl((i,\malefemale\bigr),j\Bigr),
      \end{equation}
      and $|j-s|\leq d(s,\mathbb Z\setminus I)/2$. 
      By Lemma~\ref{lem:offRbnd}, the probability that \eqref{eq:viol2} holds for given $\malefemale$, $i$, and $j$, is exponentially small in $d(j,\mathbb Z\setminus I)$. Taking a union bound implies that the probability of $(ii)$, \eqref{eq:viol}, and $s\in \Bnd(\mathcal P_I)$ being simultaneously satisfied for some $\malefemale$, $i$, and $j$, is at most
      $$
        \#I\cdot \bigl(c_5+o_q(1)\bigr)\cdot q^{c_6d(s,\mathbb Z\setminus I)},
      $$
      for some absolute constants $c_5,c_6>0$. The result follows by combining this with case $(i)$. 
      % By Lemma~\ref{lem:offRbnd}, the probability that \eqref{eq:viol2} occurs is exponentially small in $d(j,\mathbb Z\setminus I)$. 
      % Taking a union bound over $\malefemale\in \{\male,\female\}$ and $i\in I$ increases this bound by a factor of $2\# I$. 
      % Then we take a union bound over $j$ satisfying $|j-s|\leq d(s,\mathbb Z\setminus I)/2$. 
      % The sum is bounded by $c_1\# Ie^{-c_2 d(s,\mathbb Z\setminus I)}$ for some positive constants $c_1$ and $c_2$ depending on $q$ but not on $I$ or $s$, and this is an upper bound on the probability that $(ii)$ and \eqref{eq:viol} are satisfied for some $\malefemale$, $i$, and $j$. 
      % Adding this bound to that which was derived in case $(i)$ yields the result.
    \end{proof}

    \begin{proof}[Proof of Theorem~\ref{thm:exp}]
      Combine Propositions \ref{prop:main}, \ref{prop:rhoP} and \ref{prop:grBound} to
      get the convergence.

     %     \rho_q:=\mathbb P\bigl(0\in \elat( \MalPref_{q,\m Z})\bigr).
 
%      Proposition \ref{prop:moment} gives the asymptotics as $n \to \infty$. 
        If $0,1 \in \Bnd(\mathcal P)$ then both $0$ and $1$ are essential lattice cutpoints and $\Nbhd(\mathcal P,0)=\{0\}$. Thus
        $\gr(\MalPref_{q,\m Z},0)$. As $\gr(\MalPref_{q,\m Z},0)$ is bounded below by 0 and above by $\C$ we get that 
  $$\gr_q= \mathbb E(\gr(\MalPref_{q,\m Z},0))\leq 2\mathbb P(0 \not  \in \Bnd(\mathcal P))\C \leq 2(1-\rho_q)\C.$$
 By Lemma~\ref{lem:rhoP} the right hand side goes to zero. Since the left hand side is always positive it also converges to zero as $q \to 0$.
      
      It remains to prove the continuity of $\gr_q$ for $q<1$.
% By Lemma \ref{lem:resDiffBnd} for any $\gamma<1$ and $\epsilon>0$ there exists $N$ such that for all $q<1$    we have that $$\mathbb P\big(\Nbhd(\mathcal P,0) \subset \bbb{-N,N}\big)>1-\epsilon.$$ 
For any preference structure $\mathcal P$ we consider the triple $A =(a, b,\mathcal P')$ such that
$\bbb{a,b}=\Nbhd(\mathcal P,0)$ and
$\mathcal P'=\mathcal P_{\bbb{a,b}}$. Then it makes sense to write $\gr(A)$ and
$$\gr_q=\mathbb E \gr(\MalPref_{q,\m Z},0)=\sum_{A} \gr(A) \mathbb P_q(A).$$

%To see that 
%$|\gr_q-\gr_q'|<\epsilon$ if $|q-q'|<\delta$ and $q,q'<\gamma<1$ we write
Then

\begin{eqnarray*}
|\gr_q-\gr_{q'}|&=&|\mathbb E \gr(\MalPref_{q,\m Z},0)-\mathbb E \gr(\MalPref_{q',\m Z},0))|\\
&=&
\sum_{A} \gr(A) ( \mathbb P_q(A)-\mathbb P_{q'}(A) ) \\
&\leq&
\sum_{A} \gr(A) |\mathbb P_q(A)-\mathbb P_{q'}(A)|\\
&\leq&\sum_{A: \max(|a|,|b|)\leq N} \gr(A) |\mathbb P_q(A)-\mathbb P_{q'}(A)|\\
&&\hspace{1in}+
\sum_{A: \max(|a|,|b|)> N} \gr(A) |\mathbb P_q(A)-\mathbb P_{q'}(A)|.
\end{eqnarray*}
By \eqref{eq:trivBndGr}
 the first sum is at most
%$C\epsilon'$ because $\gr(\mathcal P,0)$ is uniformly bounded
$\C$ times the total variation distance between
$$ \MalPref_{q, \bbb{-N,N}} \ \ \ \text{and} \ \ \ \MalPref_{q', \bbb{-N,N}}.$$
By the continuity of the finite Mallows distributions for any $\epsilon'>0$ 
these are within $\epsilon'$ in the total variation distance if $|q-q'|$ is sufficiently small.
% the Mallows preference structure on $\bbb{-N,N}$ with parameters $q$ and $q'$.
%$ $ and $ $ is less than $\epsilon'$. The second sum is less than 
The second sum is at most 
$$2\C \sup_{q^*<\gamma} \mathbb P_{q^*}(\Nbhd(\mathcal P,0) \not \subset \bbb{-N,N}).$$
This can be made arbitrarily small by Lemma \ref{lem:resDiffBnd}. As $\epsilon$ is arbitrary this proves continuity.
    \end{proof}

\section{Conclusions and Open Problems} 
There are a number of interesting open problems in the model considered in this paper and related models. 

\begin{itemize}
\item In our model it is natural to conjecture that women are in general matches to men of a similar rank. We believe it should not be too hard to prove such a statement.
\item Similarly we believe that for the Mallows model, if one compares the highest achievable rank of men matching to a woman (achieved by the women proposing algorithm) it should be close to the lowest achievable rank of men matching to a woman 
(achieved by the men proposing algorithm).
\item It is natural to consider other permutation models that may take into account geography or the preference of individuals not to be matched below a certain threshold. 
\end{itemize} 

  \bibliography{matchingsold}

\end{document}